\documentclass[12pt]{article}

\usepackage{amsrefs}
\usepackage{amsmath,amssymb,amsthm}
\usepackage{enumerate,pxfonts,tikz}
\usepackage[all,cmtip]{xy}
\usepackage[applemac]{inputenc}
\usepackage[colorlinks=true, linkcolor=black, citecolor=black, urlcolor=black]{hyperref}

\newcommand \ab{\mathrm{ab}}
\newcommand \al{\alpha}
\newcommand \bs{\backslash}
\newcommand \C{{\mathbb C}}
\newcommand \CA{\mathcal{A}}
\newcommand \CB{\mathcal{B}}

\newcommand \CF{\mathcal{F}}

\newcommand \CN{\mathcal{N}}
\newcommand \CO{\mathcal{O}}

\newcommand \Ga{\Gamma}

\newcommand \ga{\gamma}

\newcommand \InjRad{\operatorname{InjRad}}

\newcommand \la{\lambda}

\newcommand \mqed{\tag*\qedhere}
\newcommand \N{{\mathbb N}}
\newcommand \ol{\overline}
\newcommand \om{\omega}
\newcommand \Om{\Omega}

\newcommand \PSL{\operatorname{PSL}}
\newcommand \PSO{\operatorname{PSO}}
\newcommand \Pl{\mathrm{Pl}}

\newcommand \R{{\mathbb R}}
\newcommand \sign{\operatorname{sign}}
\newcommand \SL{\operatorname{SL}}
\newcommand \sm{\smallsetminus}

\newcommand \supp{\operatorname{supp}}
\newcommand \T{{\mathbb T}}
\newcommand \tr{\operatorname{tr}}
\newcommand \vol{\operatorname{vol}}
\newcommand \VN{\operatorname{VN}}
\newcommand \what{\widehat}
\newcommand \Z{{\mathbb Z}}

\renewcommand \1{{\bf 1}}

\renewcommand \({\left(}
\renewcommand \){\right)}
\renewcommand \H{{\mathbb H}}

\newcommand{\e}
[1]{\emph{#1}\index{#1}}

\newcommand{\norm}
[1]{\left\|#1\right\|}

\renewcommand{\sp}
[1]{\left\langle #1\right\rangle}

\newcommand{\tto}
[1]{\stackrel{#1}{\longrightarrow}}

\newtheorem{theorem}{Theorem}[section]

\newtheorem{lemma}[theorem]{Lemma}

\newtheorem{proposition}[theorem]{Proposition}

\theoremstyle{definition}
\newtheorem{definition}[theorem]{Definition}
\newtheorem{example}[theorem]{Example}

\newtheorem{remark}[theorem]{Remark}

\begin{document}

\pagestyle{myheadings} \markright{BENJAMINI-SCHRAMM AND SPECTRAL CONVERGENCE}

\title{Benjamini-Schramm and spectral convergence}
\author{Anton Deitmar}
\date{}
\maketitle

{\bf Abstract:}
It is shown that under mild conditions, Benjamini-Schramm convergence of lattices in locally compact groups is equivalent to spectral convergence.
Next both notions are extended to the relative case and  then are expressed in terms of relative $L^2$-theory.

$$ $$

\tableofcontents

\newpage
\section*{Introduction}

The notion of Benjamini-Schramm convergence (BS-convergence) of graphs \cite{BS} can be extended to measured proper metric spaces.
Roughly, a sequence $(X_n)$ of such spaces converges to $X$, if for any $R>0$ the probability of a ball in $X_n$ of radius $R$ being isomorphic with a ball in $X$ converges to 1 as $n\to\infty$.

In the paper \cite{7Samurais}, BS-convergence is 
used in the context of locally symmetric Riemannian manifolds.
These are Riemannian manifolds $(M,g)$, on which the reflection $s_m$ at a given point $m\in M$, defined in a neighborhood $U$ of that point, is an isometry of the metric, see for instance \cite{Helg}.
Locally symmetric manifolds can be classified as follows:
A locally symmetric manifold $M$ has a universal covering $X$, which is a globally symmetric space, i.e., the local reflections $s_m$, $m\in M$ extend to global isometries on $X$.
The isometry group $G$ of $X$ is a Lie group and the fundamental group $\Ga$ of $M=\Ga\bs X$ is a discrete subgroup of $G$. 
The space $X$ is a product of Riemannian manifolds $X=\R^n\times X_c\times X_{nc}$.
The compact part $X_c$ is a quotient of a compact Lie group by a closed subgroup and its analysis is well understood.
The interesting part is $X_{nc}$ which has no compact or flat factors and which is of the form $G/K$ where $G$ is a semi-simple connected Lie group with finite center and $K$ a maximal compact subgroup of $G$.
Hence the most interesting locally symmetric spaces are those of the form $\Ga\bs G/K$.
It \cite{7Samurais} it is shown, among other things, that the normalised spectral measures of a uniformly discrete sequence $(\Ga_n)$ of lattices in a connected semi-simple Lie group $G$ without center, weakly converges to the Plancherel measure, if the sequence of Riemannian manifolds $\Ga_n\bs G/K$ is BS-convergent to the symmetric space $G/K$, where $K$ is a maximal compact subgroup of $G$.

In the present paper, we define BS-convergence for sequences of discrete subgroups of arbitrary locally compact groups.
We show that in the special case of semi-simple Lie groups the new notion of BS-convergence coincides with the notion in  \cite{7Samurais}.
We then show that in this generalized setting BS-convergence, together with uniform discreteness implies spectral convergence.
We also get a converse assertion, saying that if a sequence of subgroups satisfies spectral convergence, it is BS-convergent.

We further generalize both notions to the relative case, i.e., the case when all groups $\Ga_n$ have a common non-trivial normal subgroup $\Ga_\infty$.
We finally express these notions in terms of $L^2$-theory and in the last section we compute the limit measure in a concrete example.

\section{Plancherel sequences}
Throughout the paper, $G$ will denote a locally compact group.
We once and for all fix a Haar measure $dx$ on $G$.

\begin{definition}
A \e{lattice} in $G$ is a discrete subgroup $\Ga\subset G$ such that the quotient $\Ga\bs G$ carries a non-zero, $G$-invariant Radon measure $\mu$ with
$$
\mu\(\Ga\bs G\)<\infty.
$$
If $G$ admits a lattice, then $G$ is unimodular, see Theorem 9.1.6 of \cite{HA2}.
A discrete subgroup $\Ga$, such that $\Ga\bs G$ is compact, is a lattice, see Proposition 9.1.5 of \cite{HA2}.
In this case, one speaks of a \e{cocompact lattice}.
\end{definition}

In this paper, we will assume that $G$ contains a discrete subgroup $\Ga$ such that the quotient $\Ga\bs G$ is compact.
It then follows that $G$ is unimodular and that $\Ga\bs G$ carries a non-trivial $G$-invariant Radon measure, which ist finite, i.e., the group $\Ga$ is a \e{lattice}.
We then write $\mu_{\Ga\bs G}$ for the unique invariant Radon measure which is the quotient of the Haar measure on $G$ and the counting measure on $\Ga$., i.e., the unique measure which satisfies
$$
\int_Gf(x)\,dx=\int_{\Ga\bs G}\sum_{\ga\in\Ga}f(\ga x)\,d\mu_{\Ga\bs G}(x)
$$
for every $f\in L^1(G)$.
As $\mu_{\Ga\bs G}$ is derived from the Haar measure in this way, it causes no confusion when we write $dx$ instead.

\begin{definition}\label{def1.2}
For any locally compact group $G$, there is a notion of test functions, i.e., a space of functions $C_c^\infty(G)$ given by Bruhat \cite{Bruhat}, see also \cite{Tao} and \cite{TraceGroups}.
For the sake of completeness, we shall repeat it here.
First, if $L$ is a Lie group, then $C_c^\infty(L)$ is defined as the space of all infinitely differentiable functions of compact support on $L$.

Next, suppose the locally compact group $H$ has the property that $H/H^0$ is compact, where $H^0$ is the connected component.
Let $\CN$ be the family of all normal closed subgroups $N\subset H$ such that $H/N$ is a Lie group with finitely many connected components.
We call $H/N$ a \e{Lie quotient} of $H$.
Then, by \cite{MZ}, the set $\CN$ is directed by inverse inclusion and 
$$
H\cong \lim_{\substack{\leftarrow\\ N}}H/N,
$$
where the inverse limit runs over the set $\CN$.
So $H$ is a projective limit of Lie groups.
The space $C_c^\infty(H)$ is then defined to be the sum of all spaces $C_c^\infty(H/N)$ as $N$ varies in $\CN$.

Finally to the general case.
By \cite{MZ} one knows that every locally compact group $G$ has an open subgroup $H$ such that $H/H^0$ is compact, so $H$ is a projective limit of connected Lie groups in a canonical way.
A Lie quotient of $H$ then is called a \e{local Lie quotient} of $G$.
We have the notion $C_c^\infty(H)$ and for any $g\in G$ we define $C_c^\infty(gH)$ to be the set of functions $f$ on the coset $gH$ such that $x\mapsto f(gx)$ lies in $C_c^\infty(H)$.
We then define $C_c^\infty(G)$ to be the sum of all $C_c^\infty(gH)$, where $g$ varies in $G$.
Note that the definition is independent of the choice of $H$, since, given a second open group $H'$, the support of any given $f\in C_c(G)$ will only meet finitely many left cosets $gH''$ of the open subgroup $H''=H\cap H'$.
This concludes the definition of the space $C_c^\infty(G)$ of test functions.
\end{definition}

\begin{remark}\label{rm2.5}
In the sequel, we shall use the \e{trace formula}, which we now briefly explain.
For an introduction to the trace formula see Chapter 9 of \cite{HA2}.
Let $G$ be a locally compact group and let $\Ga$ be a cocompact lattice in $G$.
By $\widehat G$ we denote the \e{unitary dual} of $G$, i.e., the set of isomorphy classes of irreducible unitary representations of $G$.
On the Hilbert space $L^2(\Ga\bs G)$ one has a unitary representation of $G$ given by right translations, so for each $y\in G$ we get
\begin{align*}
R(y):L^2(\Ga\bs G)&\to L^2(\Ga\bs G),\\
R(y)\phi(x)&=\phi(xy).
\end{align*}
Since $\Ga\bs G$ is compact, it turns out, Theorem 9.2.2 of \cite{HA2}, that this representation decomposes into a direct sum of irreducible representations,
$$
(R,L^2(\Ga\bs G))\cong \bigoplus_{\pi\in\what G} N_\Ga(\pi)\pi,
$$
where the sum is a direct Hilbert sum and each class $\pi\in\what G$ occurs with finite multiplicity $N_\Ga(\pi)\in\N_0$ and only countably many $\pi\in\what G$ satisfy $N_\Ga(\pi)\ne 0$.

For $f\in C_c^\infty(G)$, integration defines  an operator $R(f)$ on the Hilbert space $L^2(\Ga\bs G)$ by
$$
R(f)\phi(x)=\int_Gf(y)\phi(xy)\,dy,\qquad \phi\in L^2(\Ga\bs G).
$$
The trace formula says that this operator is actually a trace class operator and that its trace can be computed as
$$
\tr R(f)=\int_{\Ga\bs G}\sum_{\ga\in\Ga}f(x^{-1}\ga x)\,dx=\sum_{[\ga]}\vol(\Ga_\ga\bs G_\ga)\,\CO_\ga(f).
$$
We explain the notation on the right hand side:
The sum runs over all conjugacy classes $[\ga]$ in the group $\Ga$.
For given $\ga\in \Ga$, the groups
$\Ga_\ga$ and $G_\ga$ are the centralisers of $\ga$ in $\Ga$ and $G$ respectively. 
Finally, $\CO_\ga(f)$ denotes the \e{orbital integral}
$$
\CO_\ga(f)=\int_{G_\ga\bs G}f(x^{-1}\ga x)\,dx.
$$
As the invariant measures are only unique up to scaling, to make sense of the formula, one notes that it is part of the assertion of the trace formula that for each choice of Haar measure on $G_\ga$, the volume $\vol(\Ga_\ga\bs  G_\ga)$ is finite, the orbital integral converges and the product $\vol(\Ga_\ga\bs G_\ga)\CO_\ga(f)$ does not depend on the choice of the Haar measure.
(The last assertion is easy to see, as the two factors have reciprocal dependence on scaling of the Haar measure.)

The operator $R(f)$ being trace class on $L^2(\Ga\bs G)$ immediately implies that $\pi(f)$ is trace class for each $\pi\in\what G$ with $N_\Ga(\pi)\ne 0$.
Hence the trace $\tr R(f)$ can also be spectrally expressed in terms of the traces $\tr\pi(f)$, so that the trace formula finally reads
$$
\sum_{\pi\in\what G} N_\Ga(\pi)\tr\pi(f)=\int_{\Ga\bs G}\sum_{\ga\in\Ga}f(x^{-1}\ga x)\,dx=
\sum_{[\ga]}\vol(\Ga_\ga\bs G_\ga)\,\CO_\ga(f).
$$
The middle expression is usually omitted as it easily transforms into the right hand side.
We shall, however, make use of all three expressions in the sequel.
In order to apply the trace formula, one has to find suitable test functions $f$ for which at least one of the two sides is computable via other means.
Finding such test functions requires an additional amount of work and can be a tricky business.
In \cite{HA2} one finds applications of the trace formula to the Harmonic Analysis of the Heisenberg group and the group $\SL_2(\R)$.
\end{remark}

\begin{definition}
We  say that the measure on $\what G$ given by
$$
\what\mu_\Ga=\sum_{\pi\in\what G}N_\Ga(\pi)\delta_\pi
$$
is the \e{spectral measure} attached to $\Ga$.
\end{definition}

\begin{definition}
Let $(\Ga_n)$ be a sequence of cocompact lattices in $G$.
We say that the sequence is a \e{Plancherel sequence}, if for every $f\in C_c^\infty(G)$ we have
$$
\frac1{\vol(\Ga_n\bs G)}\int_{\what G}\hat f(\pi)\,d\what\mu_{\Ga_n}(\pi)\ \longrightarrow f(e)
$$
as $n\to\infty$, where $\hat f(\pi)=\tr\pi(f)$.

If the group $G$ is type I, then $f(e)=\int_{\what G}\hat f(\pi)\,d\what\mu_\Pl(\pi)$, where $\what\mu_\Pl$ is the Plancherel measure on $\what G$, so that in this case the sequence $(\Ga_n)$ is Plancherel if and only if in the dual space of $C_c^\infty(G)$ one has weak-*-convergence
$$
\frac1{\vol(\Ga_n\bs G)}\what\mu_{\Ga_n}\ \longrightarrow\ \what\mu_\Pl.
$$
\end{definition}

\begin{remark}
Assume that $G$ is a finite product $\prod_{i=1}^n G_i$, where each $G_i$ is the group of rational points of a linear reductive group ${\bf G}_i$ over a local field $F_i$ of characteristic zero.
If a sequence $(\Ga_n)$ of  lattices is a Plancherel sequence, then
$$
\frac1{\vol(G/\Ga_n)}\what\mu_{\Ga_n}(U)\ \longrightarrow\ \what\mu_\Pl(U)
$$
for every open set $U\subset\what G$, which is $\what\mu_\Pl$-regular and relatively compact.
This follows from the density principle of Sauvageot \cite{Sauvageot}.
\end{remark}

\begin{example}
 Except for trivial examples like finite groups, for a sequence to be Plancherel, it will be necessary that $\vol(G/\Ga_n)$ tends to infinity.
This is not sufficient, though, as we  see by the example of $G=\R^2$ and
$\Ga_n=n^2\Z\times\frac1n\Z$.
Another example of this can be constructed in the group $G=\SL_2(\R)$, where one chooses $\Ga_n$ in the Teichm\"uller space such that the hyperbolic manifold $\Ga_n\bs\H$ has a very short closed geodesic, but big volume. 
\end{example}

\section{Generalized Benjamini-Schramm convergence}\label{Benjamini-Schramm}

For any subset $S\subset G$ we write $S^\star$ for the set $S\sm\{1\}$.

\begin{definition}
We say that a sequence $(\Ga_n)$ of lattices in $G$ is \e{Benjamini-Schramm convergent} or \e{BS-convergent} to $\{1\}$, if for every compact set $C\subset G$ the sequence
$$
P_{\Ga_n\bs G}\(\big\{x\in \Ga_n\bs G: x^{-1}\Ga_n^\star x\cap C\ne \emptyset\big\}\)
$$
tends to zero,  where $P_{\Ga_n\bs G}$ is the normalised invariant volume on $\Ga_n\bs G$.
\end{definition}

Next we check the compatibility of this definition with the definition of BS-convergence in \cite{7Samurais}.

\begin{definition}
Let $(M,g)$ be a Riemannian manifold and let $x\in M$ a point.
In Differential Geometry, one defines the \e{exponential map at $x$} as a map $\exp_x:T_xM\to M$ from the tangent space $T_xM$ at the point $x$ to the manifold as follows:
Every $v\in T_xM$, which is not zero, defines a unique geodesic tangential to $v$ and starting at $x$. One follows that geodesic for the time $\norm v_g$ to reach a point one calls $\exp_x(v)$.
Finally, one sets $\exp_x(0)=x$.
The \e{injectivity radius} of $M$ at the point $x$, denoted $\InjRad(x)$, is the supremum of all $r>0$ such that $\exp_x$ is injective on the open ball $\big\{ v\in T_xN: \norm v_g<r\big\}$. The injectivity radius  is  $>0$ for every point $x\in M$.
\end{definition}

\begin{definition}
Let $G$ be a connected  semi-simple Lie group with trivial center and let $(\Ga_n)$ be a sequence of lattices in $G$.
Fix a maximal compact subgroup $K$.
Then the space $G/K$ carries a natural $G$-invariant Riemannian metric given by the Killing form on the Lie algebra of $G$, see \cite{Helg}.
In \cite{7Samurais} the sequence of spaces $\Ga_n\bs X$ is said to be \e{Benjamini-Schramm convergent} or \e{BS-convergent} to $X$ if for every $R>0$
$$
P_{\Ga_n\bs X}\(\big\{x\in\Ga_n\bs X: \InjRad(x)\le R\big\}\)
$$
tends to zero, where $\InjRad(x)$ is the injectivity radius at the point $x$.
\end{definition}

\begin{proposition}\label{prop2.4}
Let $G$ be a connected  semi-simple Lie group with trivial center and $(\Ga_n)$ a sequence of lattices in $G$.
Then $(\Ga_n)$ is Benjamini-Schramm convergent to $\{1\}$ if and only if $\Ga_n\bs X$ is Benjamini-Schramm convergent to $X$.
\end{proposition}

\begin{proof}
Suppose that $(\Ga_n)$ is BS-convergent to $\{1\}$ and let $R>0$.
Let $\ol B_R$ be the closed ball around $eK$ of radius $R$ and let $U_R\subset G$ be its pre-image under $G\to G/K=X$.
Let $x\in G$.
The condition $\InjRad(xK)\le R$ is equivalent to the existence of $u_1,u_2\in U_R$ with $u_1K\ne u_2K$ such that
$$
\Ga_n xu_1 K=\Ga_n xu_2K,
$$
or $xu_1=\ga xu_2k$ for some $\ga\in\Ga_n^\star$, $k\in K$, or
$$
x^{-1}\ga x=u_1k^{-1}u_2^{-1}.
$$
So $\InjRad(xK)\le R$ if and only if
$$
x^{-1}\Ga_n^\star x\cap U_RKU_R^{-1}\ne \emptyset.
$$
Since $U_RKU_R^{-1}$ is compact, it follows that $\Ga_n\bs X$ is BS-convergent to $X$.

For the converse note that for every compact set $C\subset G$ there exists $R>0$ such that $C\subset U_R KU_R^{-1}$.
\end{proof}

\begin{definition}
Let $G$ be a locally compact group.
A sequence $(\Ga_n)$ of subgroups is called \e{uniformly discrete}, if there exists a unit-neighborhood $U$ such that $x^{-1}\Ga_nx\cap U=\{1\}$ holds for every $n$ and every $x\in G$.
\end{definition}

In \cite{7Samurais} it is shown, that a sequence of cocompact lattices $(\Ga_n)$ in a semi-simple Lie group $G$ is Plancherel, if it is uniformly discrete and $\Ga_n\bs G/K$ is Benjamini-Schramm convergent to $G/K$.
We shall generalize this result to arbitrary locally compact groups and we shall also prove the converse statement.

\begin{theorem}
Let $G$ be a locally compact group and let $(\Ga_n)$ be a sequence of cocompact lattices which is uniformly discrete.
Then the following are equivalent:
\begin{enumerate}[\rm (a)]
\item  $(\Ga_n)$ is BS-convergent to $\{1\}$.
\item
$(\Ga_n)$ is a Plancherel sequence.
\end{enumerate}
The implication (b)$\Rightarrow$(a) even holds for any sequence of cocompact lattices, i.e., without the assumption of $(\Ga_n)$ being uniformly discrete.
\end{theorem}

In the case of a semi-simple Lie group the proof of this theorem, in the light of Proposition \ref{thm3.3}, gives a new proof of Theorem 6.7 of \cite{7Samurais} which avoids the use of the Chabauty space.

\begin{proof}
We first show a lemma which is of independent interest.
We shall use the notation  $\# M$ for the cardinality of a set $M$.

\begin{lemma}\label{lem2.4}
Let $(\Ga_n)$ be a sequence of cocompact lattices which is uniformly discrete.
Then for any compact set $C\subset G$ there exists a uniform bound $r\in\N$ on the cardinality as 
$$
\# \(x^{-1}\Ga_n^\star x\cap C\)\ \le\ r
$$
holds for all $n\in\N$ and all $x\in G$.
\end{lemma}

\begin{proof}
Let $\emptyset\ne C\subset G$ be compact and let $U$ be a relatively compact, open unit-neighborhood such that $x^{-1}\Ga_nx\cap U=\{1\}$ for all $n\in\N$, $x\in G$.
Let $V$ be a symmetric open unit neighborhood with $V^2\subset U$.
Then for all $x,y\in G$, $n\in\N$ we have $\#\(x^{-1}\Ga_nx\cap yV\)\le 1$, since for two
$\sigma_1=yv_1$ and $\sigma_2=yv_2$ in $x^{-1}\Ga_nx\cap yV$ we have $\sigma=\sigma_2^{-1}\sigma=v_2^{-1}v_1\in x^{-1}\Ga_nx\cap V^2\subset x^{-1}\Ga_nx\cap U=\{1\}$, so $\sigma_1=\sigma_2$.

As $C$ is compact, there are $x_1,\dots,x_r\in G$ such that 
$$
C\subset x_1 V\cup\dots\cup x_rV.
$$
Every group of the form $x^{-1}\Ga_nx$ intersects each $x_iV$ in at most one element, hence intersects $C$ in at most $r$ elements.
\end{proof}

Now for the proof of (a)$\Rightarrow$(b) in the theorem:
Let $f\in C_c^\infty(G)$, let 
$C=\supp(f)$ and let $r\in\N$ be as in the lemma.

We have to show that the expression
$$
A_n= \frac1{\vol(G/\Ga_n)}\int_{\what G}\tr\pi(f)\,d\mu_{\Ga_n}-f(e).
$$
tends to zero as $n\to\infty$.
We now use the trace formula, see Remark \ref{rm2.5}.
The trace formula implies that
\begin{align*}
|A_n|&=\left|\int_{\Ga_n\bs G}\sum_{\ga\in\Ga_n^\star}f(x^{-1}\ga x)\,dP_{\Ga_n\bs G}(x)\right|\\
&\le r\norm f_\infty P_{\Ga_n\bs G}\(\big\{x\in \Ga_n\bs G: x^{-1}\Ga_n^\star x\cap C\ne \emptyset\big\}\).
\end{align*}
This tends to zero by assumption.

Finally for (b)$\Rightarrow$(a):
Let $C\subset G$ be a compact set and let $f\in C_c^\infty(G)$ with $f\ge \1_C$.
Then $A_n\to 0$ with
\begin{align*}
A_n&=\int_{\Ga_n\bs G}\sum_{\ga\in\Ga_n^\star}f(x^{-1}\ga x)\,dP_{\Ga_n\bs G}(x)\\
&\ge \int_{\Ga_n\bs G}\#\(x^{-1}\Ga_n^\star x\cap C\)\,dP_{\Ga_n\bs G}(x)\\
&\ge P_{\Ga_n\bs G}\(\big\{x\in \Ga_n\bs G: x^{-1}\Ga_n^\star x\cap C\ne \emptyset\big\}\).\mqed
\end{align*}
\end{proof}

\begin{definition}
For $x\in\R$ set
$$
\sign(x)=\begin{cases} 1&x>0,\\
-1& x<0,\\
0&x=0.
\end{cases}
$$
\end{definition}

\begin{proposition}\label{prop2.9}
\begin{enumerate}[\rm (a)]
\item A given sequence $(\Ga_n)$ of cocompact lattices is Plancherel if and only if for every compact set $C\subset G$ the sequence
$$
\int_{\Ga_n\bs G}\#\(x^{-1}\Ga_n^\star x\cap C\)\,dP_{\Ga_n\bs G}(x)
$$
tends to zero as $n\to\infty$.
\item A given sequence $(\Ga_n)$ of cocompact lattices is BS convergent to $\{1\}$ if and only if for every compact set $C\subset G$ the sequence
$$
\int_{\Ga_n\bs G}\sign\(\#\(x^{-1}\Ga_n^\star x\cap C\)\)\,dP_{\Ga_n\bs G}(x)
$$
tends to zero as $n\to\infty$.
\end{enumerate}
\end{proposition}

\begin{proof}
(a)
Let $(\Ga_n)$ be a Plancherel sequence and let $C\subset G$ be a compact set.
We want to show that $B_n=\int_{\Ga_n\bs G}\#\(x^{-1}\Ga_n^\star x\cap C\)\,dP_{\Ga_n\bs G}(x)
$ tends to zero as $n\to\infty$.
By enlarging $C$ we can assume $e\in C$.
Let $f\in C_c^\infty(G)$ with $f\ge \1_C$ and $f(e)=1$.
Then 
$$
\frac1{\vol(\Ga_n\bs G)}\sum_{\pi\in\what G} N_\Ga(\pi)\,\tr\pi(f)
$$
tends to $f(e)=1$ as $n\to\infty$.
By the trace formula we have
\begin{align*}
1\le B_n+1&=\int_{\Ga_n\bs G}\#\(x^{-1}\Ga_n x\cap C\)\,dP_{\Ga_n\bs G}(x)\\
&\le \int_{\Ga_n\bs G}\sum_{\ga\in\Ga_n}f(x^{-1}\ga x)\,dP_{\Ga_n\bs G}(x)\\
&=\frac1{\vol(\Ga_n\bs G)}\sum_{\pi\in\what G} N_\Ga(\pi)\,\tr\pi(f).
\end{align*}
As the latter tends to $1$ for $n\to\infty$, we get $B_n\to 0$ as claimed.

For the converse direction assume that $(\Ga_n)$ is a sequence of cocompact lattices such that for any compact set $C\subset G$ the sequence 
$$
\int_{\Ga_n\bs G}\#\(x^{-1}\Ga_n^\star x\cap C\)\,dP_{\Ga_n\bs G}(x)
$$ 
tends to zero as $n\to\infty$.
Let $f\in C_c^\infty$ and let $C$ be its support. Further let $M>0$ be such that $|f(x)|\le M$ for every $x\in G$.
Then, again by the trace formula,
\begin{align*}
\left|\frac1{\vol(\Ga_n\bs G)}\int_{\what G}\hat f(\pi)\,d\what\mu_{\ga_n}(\pi)-f(e)\right|
&=\left|\frac1{\vol(\Ga_n\bs G)} \sum_{\pi\in\what G} N_{\Ga_n}(\pi)\,\tr\pi(f)-f(e)\right|\\
&= \left|\frac1{\vol(\Ga_n\bs G)}\int_{\Ga_n\bs G}\sum_{\ga\in\Ga_n^\star}f(x^{-1}\ga x)\,dx\right|\\
&\le \frac1{\vol(\Ga_n\bs G)}\int_{\Ga_n\bs G}\sum_{\ga\in\Ga_n^\star}|f(x^{-1}\ga x)|\,dx\\
&\le \frac M{\vol(\Ga_n\bs G)}\int_{\Ga_n\bs G}\#\( x^{-1}\Ga_n^\star x\cap C\) \,dx.
\end{align*}
As this tends to zero we conclude the claim.

Part (b) follows from the definition.
\end{proof}

\begin{example}
To round things up, we  give an example of a sequence $(\Ga_n)$ of cocompact lattices, which is Plancherel, but not uniformly discrete.
The group $G$ will be $\PSL_2(\R)$, which we view as the group of orientation-preserving isometries of the hyperbolic plane $\H$.
For given $n\in\N$ fix a compact hyperbolic surface $\Ga\bs\H$ whose shortest closed geodesic $c_1$ has length $<\frac1n$.
Let $\ga_1\in\Ga$ be an element in the homotopy class of $c_1$, i.e., the deck transformation $\ga_1$ closes the geodesic $c_1$. 
Then there are $\ga_2,\ga_3,\dots,\ga_{2g}\in\Ga$, where $g$ is the genus, such that $\Ga$ is generated by $\ga_1,\dots,\ga_{2g}$ with the only relation $[\ga_1,\ga_2]\cdots[\ga_{2g-1},\ga_{2g}]=1$.
The largest abelian quotient $\Ga^\ab\cong H_1(\Ga\bs\H,\Z)$ is freely generated by the classes $[\ga_1],\dots,[\ga_{2g}]$.
Let $p:\Ga\to\Z[\ga_1]\oplus\dots\oplus\Z[\ga_{2g}]$ be the quotient map.
For $N\in\N$, let 
$$
\Ga(N)=p^{-1}\(\Z[\ga_1]\oplus N\(\Z[\ga_2]\oplus\dots\oplus \Z[\ga_{2g}]\)\).
$$
For given $R>0$ there exists $N_0$ such that for all $N\ge N_0$ the only closed geodesics of length $\le 2R$ in $\Ga(N)\bs \H$ are multiples of $c_1$.
This implies that for a given compact set $C\subset G$ and $N$ large enough the set of all $x\in G$ with $x^{-1}\Ga(N)x\cap C\ne\emptyset$ lies in the pre-image under $x\mapsto \Ga(N)xK\in\Ga(N)\bs\H$ of a tubular neighborhood of fixed size of the geodesic $c_1$.
Increasing $N$, one thus gets 
$$
\int_{\Ga(N)\bs G}\#\(x^{-1}\Ga(N)^\star x\cap C\)\,dP_{\Ga(N)\bs G}(x)<\frac1n.
$$
Setting $\Ga_n$ to be equal to this $\Ga(N)$ gives the desired example.
\end{example}

\section{The relative case and $L^2$-theory}
In this section we consider the following situation:
$(\Ga_n)_{n\in\N}$ is a sequence of cocompact discrete subgroups of a locally compact group $G$, and $\Ga_\infty$ is a common normal subgroup, i.e., $\Ga_\infty\triangleleft \Ga_n$ for every $n$.

\begin{definition}
We say that the sequence $(\Ga_n)$ \e{Plancherel converges} to $\Ga_\infty$, written $\Ga_n\tto{Pl}\Ga_\infty$, if for every compact set $C\subset G$ the sequence
$$
\int_{\Ga_n\bs G}\#\(x^{-1}\(\Ga_n\sm\Ga_\infty\)x\cap C\)\,dP_{\Ga_n\bs G}(x)
$$
tends to zero as $n\to\infty$.
This definition is inspired by Proposition \ref{prop2.9}.
\end{definition}

\begin{definition}
We say that the sequence $(\Ga_n)$ \e{Benjamini-Schramm converges} to $\Ga_\infty$, written $\Ga_n\tto{BS}\Ga_\infty$, if
for every compact set $C\subset G$ the sequence
$$
\int_{\Ga_n\bs G}\sign\(\#\(x^{-1}\(\Ga_n\sm\Ga_\infty\)x\cap C\)\)\,dP_{\Ga_n\bs G}(x)
$$
tends to zero as $n\to\infty$.

In the special case $\Ga_\infty=\{1\}$, these notions coincide with the notion of the previous section.
\end{definition}

\begin{theorem}\label{thm3.3}
Let $(\Ga_n)$ be a sequence of cocompact lattices which is uniformly discrete.
Let $\Ga_\infty$ be a common normal subgroup.
Then the following are equivalent:
\begin{enumerate}[\rm (a)]
\item  $(\Ga_n)$ is BS-convergent to $\Ga_\infty$.
\item
$(\Ga_n)$ is Plancherel convergent to $\Ga_\infty$.
\end{enumerate}
The implication (b)$\Rightarrow$(a) even holds for any sequence of cocompact lattices, i.e., without the assumption of $(\Ga_n)$ being uniformly discrete.
\end{theorem}

\begin{proof}
(a)$\Rightarrow$(b): Given a compact $C\subset G$, by Lemma \ref{lem2.4} there exists $r\in\N$ such that
$$
\# \(x^{-1}\Ga_n x\cap C\)\ \le\ r
$$
holds for all $x\in G$ and all $n\in\N$.
Therefore
\begin{align*}
&\int_{\Ga_n\bs G}\#\(x^{-1}\(\Ga_n\sm\Ga_\infty\)x\cap C\)\,dP_{\ga_n\bs G}(x)\\
&\le r \int_{\Ga_n\bs G}\sign\(\#\(x^{-1}\(\Ga_n\sm\Ga_\infty\)x\cap C\)\)\,dP_{\ga_n\bs G}(x),
\end{align*}
which implies the claim. The converse direction is trivial.
\end{proof}

\begin{definition}
For a discrete subgroup $\Ga\subset G$ a \e{fundamental domain} is an open set $\CF\subset G$ such that 
\begin{itemize}
\item $\CF\cap \ga\CF=\emptyset$ holds for every $\ga\in\Ga^\star$,
\item $\Ga\ol\CF=G$, where $\ol\CF$ is the topological closure of $\CF$ and
\item Let $\partial \CF=\ol\CF\sm\CF$ denote the boundary of $\CF$.
Then the boundary is a \e{local null-set}, i.e., 
 for every compact set $C\subset G$ the set $C\cap\Ga\partial\CF$ has Haar measure zero.
\end{itemize}
\end{definition}

\begin{proposition}
\begin{enumerate}[\rm (a)]
\item For every locally compact group $G$ and every discrete subgroup $\Ga$ there exists a fundamental domain.
\item If $\CF$ is a fundamental domain for $\Ga\bs G$ then for every $1\le p<\infty$ the natural map $L^p(\Ga\bs G)\to L^p(\CF)$ is an isomorphism of Banach spaces.
\end{enumerate}
\end{proposition}

\begin{proof}
(a)
We use the classification results on locally compact groups as cited in Definition \ref{def1.2}.
First we consider the case of a Lie group $G$.
Then any discrete subgroup $\Ga\subset G$ is countable.
By Choosing an inner product on the tangent space $T_eG$ we get a left-invariant Riemann metric on the smooth manifold $G$.
Let $d$ denote the corresponding distance function.
For any $x\in G$ and any $R>0$ the closed ball $B_R(x)$ of radius $R$ around $x$ then is compact and hence $B_P(x)\cap \Ga$ is finite.
Let 
$$
\CF=\big\{ x\in G: d(x,e)<d(x,\ga)\ \forall_{\ga\in\Ga^\star}\big\}.
$$
A simple application of the triangle inequality yields
$$
\CF=\left\{ x\in G: d(x,e)<d(x,\ga)\ \forall_{\substack{\ga\in\Ga^\star\\ d(\ga,e)<3d(x,e)}}\right\}.
$$
So locally, only finitely many inequalities matter, which implies that $\CF$ is open.
We claim that $\CF$ is a fundamental domain.
For each $\ga\in\Ga^\star$ the set
$$
S_\ga=\big\{ x\in G: d(x,e)=d(x,\ga)\big\}
$$
is the zero set of the function $d(x,e)-d(x,\ga)$, which is smooth in a neighborhood of each of its zeros. Therefore, the set $S_\ga$ is contained in a countable union of sub-manifolds and hence of Haar measure zero.
The closure of $\CF$ equals
$$
\ol\CF=\big\{ x\in G: d(x,e)\le d(x,\ga)\ \forall_{\ga\in\Ga^\star}\big\}
$$
and so $\partial\CF$ is contained in the countable union of all $S_\ga$ and therefore is of measure zero.
As $\Ga$ is countable, $\Ga\partial\CF$ is a null-set as well.
The properties $\CF\cap\ga\CF=\emptyset$ for $\ga\ne e$ and $\ol\Ga= G$ are easily verified.
So $\CF$ is a fundamental domain.

Next we assume that $G/G^0$ is compact, where $G^0$ is the connected component of $G$.
Then $G$  is the projective limit
$$
G=\lim_{\substack{\leftarrow\\ \CN}}G/\CN
$$
of all its Lie quotients $G/\CN$.
Note that in this projective system all structure maps are continuous, surjective and open.
An open set $U\subset G$ is called \e{pure}, if $U=p_\CN^{-1}(V)$ for some open set $V\subset G/\CN$, where $G/\CN$ is a Lie quotient of $G$ and $p_\CN$ is the projection $G\to G/\CN$.
As $G$ carries the projective topology, any open set in $G$ is a union of pure open sets.
Let $\Ga\subset G$ be a discrete subgroup.
Then there exists a pure open set $\CF=p_{\CN_0}^{-1}(V)$ such that $\Ga\cap \CF=\{ e\}$.
The group $\Ga_{\CN_0}=p_{\CN_0}(\Ga)$ is a discrete subgroup of $G/\CN_0$.
We can  choose $V$ to be a fundamental domain with respect to $\Ga_{\CN_0}$ and we claim that then $\CF$ is one for $\Ga$.
It is open and since $\Ga\cap \CF=\{ e\}$ the map $\Ga\to\Ga_{\CN_0}$ is injective.
So let $\ga\in\Ga^\star$, then $p_{\CN_0}(\ga)\in\Ga_{\CN_0}^\star$ and therefore $V\cap p_{\CN_0}(\ga)V=\emptyset$. Taking pre-images under $p_{\CN_0}$ we arrive at $\CF\cap \ga\CF=\emptyset$.

In order to show $\Ga\ol\CF=G$ we let $x\in G$.
Replacing $x$ by $\ga x$ for some $\ga\in\Ga$ we may assume that $p_{\CN_0}(x)\in\ol V$, or $x\in p_{\CN_0}^{-1}(\ol V)$.
The latter set is closed and contains $\CF$, therefore $p_{\CN_0}^{-1}(\ol V)\supset\ol\CF$.
As $\CF$ is stable under $\CN_0$-translations, so is $\ol\CF$. Therefore
$$
\ol\CF=p_{\CN_0}^{-1}(p_{\CN_0}(\ol\CF))
$$
and
$p_{\CN_0}(\ol\CF)$ is a closed subset of $G/\CN_0$ which contains $V$, hence $\ol V$ and we conclude
$\ol\CF=p_{\CN_0}^{-1}(\ol V)$, which means we may assume $x\in\ol\CF$ after a $\Ga$-translation, so indeed $\Ga\ol\CF=G$.
By the same token we get $\partial\CF=p_{\CN_0}^{-1}(\partial V)$. In the projective limit topology, any set of the form $p_{\CN_0}^{-1}(C)$ is compact, when $C\subset G/\CN_0$ is compact.
Hence the Haar measure on $G$ induces a Haar measure on $G/\CN_0$ and therefore $\partial\CF$ is a set of Haar measure zero.
This finishes the case of $G/G^0$ being compact.

Finally let $G$ be arbitrary.
Then there exists an open subgroup $H$ such that $H/H^0$ is compact.
Then $G$ is a disjoint  union of open $H$-cosets, $G=\bigsqcup_{j\in J}g_jH$. 
Let $\Ga_H=\Ga\cap H$, then, likewise, $\Ga=\bigsqcup_{i\in I}\ga_i\Ga_H$.
We can assume $I\subset J$ and $\ga_i=g_i$ for every $i\in I$.
Let $\CF_H$ be a fundamental domain in $H$ for the discrete subgroup $\Ga_H$.
We get
$$
\Ga \ol \CF_H=\bigsqcup_{i\in I}\ga_i H=\Ga H.
$$
Let $g\in G$ be such that $\Ga g H\ne \Ga H$.
Then $\Ga g H=\Ga (gHg^{-1})g=\Ga \ol\CF_{gHg^{-1}} g$, where $\CF_{gHg^{-1}}\subset gHg^{-1}$ is a fundamental domain for $\Ga_{gHg^{-1}}(gHg^{-1})$.
Let $(y_\al)_{\al\in A}$ be a set of representatives of $\Ga\bs (G\sm \Ga H)/H$, i.e, 
$$
G\sm \Ga H=\bigsqcup_{\al\in A}\Ga y_\al H.
$$
Accordingly we set
$$
\CF=\CF_H\sqcup\bigsqcup_{\al\in A}(\CF_{y_\al H y_\al^{-1}})y_\al.
$$
Then $\CF$ is open, we have $\Ga\ol\CF=G$ and $\CF\cap\ga\CF=\emptyset$ for $\ga\in\Ga^\star$.
Finally, 
$$
\Ga\partial\CF=\Ga\partial\CF_H\sqcup\bigsqcup_{\al\in A}\(\Ga\partial\CF_{y_\al H y_\al^{-1}}\)y_\al.
$$
The $\Ga$-orbits of boundaries on the right hand side of this formula are local null-sets and as any given compact set $C\subset G$ meets only finitely many $H$ cosets,
the set $C\cap \Ga\partial\CF$ is a null-set.
The proof of part (a) is finished.

(b) Let $f:\Ga\bs G\to [0,\infty)$ be an integrable function.
We can view $f$ also as a function on $G$, which is $\Ga$-invariant.
All we need to show is
$$
\int_{\Ga\bs G}f(x)\,dx=\int_\CF f(x)\,dx.
$$
Let $p(\CF)$ denote the image of $\CF$ in $\Ga\bs G$.
Then clearly
$$
\int_{p(\CF)}f(x)\,dx=\int_\CF f(x)\,dx.
$$
Let $\phi\ge 0$ be a continuous function on $G$ such that for every compact set $C\subset G$ the set of all $\ga\in\Ga$ such that $\phi(\ga C)\ne \{0\}$ is finite and such that the continuous $\Ga$-invariant function
$$
\phi^\Ga(x)=\sum_{\ga\in \Ga}\phi(\ga x)
$$
has no zeros in $G$.
One way to construct such a function is to start with function only satisfying the first requirement and then apply Zorns lemma on the support of $\phi^\Ga$.
Then $1/\phi^\Ga$ is a continuous function and we set
$$
\psi(x)=\frac{\phi(x)}{\phi^\Ga(x)}.
$$
Then $\psi^\Ga=1$.
Hence the function $f\psi$ is integrable on $G$. By Corollary 1.3.6 (d) of \cite{HA2} the function $f\psi$ is identically zero outside a $\sigma$-compact open subgroup $H$ of $G$.
In  particular, the set
$$
\supp(f\psi)\cap\Ga\partial\CF
$$
is a null-set.
Let $G^\circ=G\sm\Ga\partial\CF$, then
$$
f(x)=\sum_{\ga\in\Ga}(f\psi)(\ga x)
$$
and hence
\begin{align*}
\int_{\Ga\bs G}f(x)\,dx
&=\int_{\Ga\bs G}\sum_{\ga\in\Ga}(f\psi)(\ga x)\,dx
=\int_Gf(x)\psi(x)\,dx\\
&=\int_{G^\circ}f(x)\psi(x)\,dx
=\int_{p(\CF)}\sum_{\ga\in\Ga}(f\psi)(\ga x)\,dx\\
&=\int_{p(\CF)}f(x)\,dx
=\int_{\CF}f(x)\,dx
\end{align*}
as claimed.
\end{proof}

\begin{example}
Consider the group $G=\PSL_2(\R)$ and let $\Ga\subset G$ be a cocompact, torsion free, discrete subgroup.
Let $g\ge 2$ be the genus of the Riemann surface $\Ga\bs\H$, where $\H$ is the upper half plane, which can be identified with $G/K$, where $K=\PSO(2)$.
Then the greatest abelian quotient $\Ga^\ab$ of $\Ga$ is isomorphic with $H_1(\Ga)\cong\Z^{2g}$. 
Let $\chi:\Ga\twoheadrightarrow\Z^{2g}\twoheadrightarrow A$ be a surjective homomorphism to some infinite, torsion-free, abelian group $A$.
For any $n\in\N$ let
$$
\Ga_n=\chi^{-1}(nA).
$$
Then the sequence $(\Ga_n)$ is uniformly discrete and converges, in the sense of Theorem \ref{thm3.3}, to
$$
\Ga_\infty=\ker(\chi).
$$  
\end{example}

For each $n$, fix a fundamental domain $\CF_n\subset G$ for $\Ga_n\bs G$.
Then there is a canonical isomorphism $L^2(\Ga_n\bs G)\cong L^2(\CF_n)$.
Fix a set $V_n\subset \Ga_n$ of representatives for $\Ga_n/\Ga_\infty$, then $\CF_\infty=\bigcup_{v\in V_n}v\CF_n$ is a fundamental set for $\Ga_\infty$ and so we get an isomorphism
$$
L^2(\Ga_\infty\bs G)\cong L^2(\Ga_n\bs G)\ \what\otimes\ \ell^2(\Ga_n/\Ga_\infty).
$$
As $\Ga_\infty$ is normal in $\Ga_n$, the left translation yields a representation of $\Ga_n/\Ga_\infty$ on $L^2(\Ga_\infty\bs G)$ by
$$
L_\ga\phi(\Ga_\infty x)=\phi(\Ga_\infty \ga^{-1} x),\qquad \phi\in L^2(\Ga_\infty\bs G).
$$
Let $\CA_n\subset\CB\(L^2(\Ga_\infty\bs G)\)$ be the von Neumann algebra defined as the commutant of this action, i.e., $\CA_n=L(\Ga_n/\Ga_\infty)^\circ$.

\begin{lemma}
In the above isomorphism $\Psi:L^2(\Ga_\infty\bs G)\to L^2(\Ga_n\bs G)\ \what\otimes\ \ell^2(\Ga_n/\Ga_\infty)$ the representation $L$ transforms to the left translation action on the factor $\ell^2(\Ga_n/\Ga_\infty)$.
\end{lemma}

\begin{proof}
This is standard, but we include a proof for the convenience of the reader.
We need to make explicit the isomorphism above. For this note that for every $x\in G$ there exists a unique $\sigma_{n,x}\in\Ga_n$ such that $\sigma_{n,x}x\in\CF_n$.
Note that the map $G\to\CF_n$, $x\mapsto\sigma_{n,x}x$ is $\Ga_n$-invariant.
We have
$$
\Psi(\phi)=\sum_{v\in V_n}\phi_v\otimes\delta_v,
$$
where $\phi_v(x)=\phi(v\sigma_{n,x}x)$ for $x\in G$ and $\delta_v(\ga\Ga_\infty)=1$ if $\ga\Ga_\infty=v\Ga_\infty$ and zero otherwise. As $\phi$ is left invariant under $\Ga_\infty$, the function $\phi_v$ does not depend on the choice or $v$, i.e., the choice of $V_n$.
Hence we can write $\Psi(\phi)=\sum_{\ga\in\Ga_n/\Ga_\infty}\phi_\ga\otimes\delta_\ga$.
For $\ga_0\in\Ga_n/\Ga_\infty$ we have
\begin{align*}
\Psi(L_{\ga_0}\phi)&=\sum_{\ga\in\Ga_n/\Ga_\infty}(L_{\ga_0}\phi)_\ga\otimes\delta_\ga,
\end{align*}
and $(L_{\ga_0}\phi)_\ga(x)=L_{\ga_0}\phi(\ga\sigma_{n,x}x)=\phi(\ga_0^{-1}\ga\sigma_{n,x}x)=\phi_{\ga_0^{-1}\ga}(x)$, so that
\begin{align*}
\Psi(L_{\ga_0}\phi)&=\sum_{\ga\in\Ga_n/\Ga_\infty}(\phi)_\ga\otimes\delta_{\ga_0\ga}\\
&=\sum_{\ga\in\Ga_n/\Ga_\infty}(\phi)_\ga\otimes L_{\ga_0}\delta_{\ga}.\mqed
\end{align*}
\end{proof}

The lemma implies that
$$
\CA_n\cong \CB(L^2(\Ga_n\bs G))\otimes \VN(\Ga_n/\Ga_\infty),
$$
where $\VN(\Ga_n/\Ga_\infty)$ is the group  von Neumann algebra of $\Ga_n/\Ga_\infty$, which is defined as the commutant of the left translation on 
$\ell^2(\Ga_n/\Ga_\infty)$ and is generated by  right translations
$$
R_\ga f(x)=f(x\ga),\qquad f\in\ell^2(\Ga_n/\Ga_\infty).
$$
Let $\tau_n:\VN(\Ga_n/\Ga_\infty)\to\C$,
$$
\tau_n\(\sum_{\ga\in\Ga_n/\Ga_\infty}c_\ga R_\ga\)=c_e,
$$
where $e$ is the neutral element of the group $\Ga_n/\Ga_\infty$.
Then $\tau_n$ is a finite trace on $\VN(\Ga_n/\Ga_\infty)$. Let $\tr_{L^2(\Ga_n\bs G)}$ be the standard trace on $\CB(L^2(\Ga_n\bs G))$. Then
$$
\tr_{\Ga_n,\Ga_\infty}^{(2)}=\frac1{\vol(\Ga_n\bs G)}\tr_{L^2(\Ga_n\bs G)}\otimes \tau_n
$$ 
is the normalised \e{$L^2$-trace} on $\CA_n$.

For $f\in C_c^\infty(G)$ we write $R_{\Ga_n}(f)$ for the induced operator on $L^2(\Ga_n\bs G)$ given by
$$
R_{\Ga_n}(f)\phi(x)=\int_Gf(y)\phi(xy)\,dy.
$$

\begin{lemma}\label{lem4.2}
For $f\in C_c^\infty(G)$ the operator $R_{\Ga_n}(f)$ lies in $\CA_n$. It has a well-defined $L^2$-trace which equals
$$
\tr_{\Ga_n,\Ga_\infty}^{(2)}(R_{\Ga_\infty}(f))=\frac1{\vol(\Ga_n\bs G)}\sum_{[\ga]_{\Ga_n}\subset \Ga_\infty}
\vol\(\Ga_{n,\ga}\bs G_\ga\)\,\CO_\ga(f),
$$
where the sum runs over the $\Ga_n$ conjugacy classes $[\ga]_{\Ga_n}$ of elements $\ga\in\Ga_\infty$.

If there exists a lattice $\Ga$ such that every $\Ga_n$ is a normal subgroup of $\Ga$, then the expression $\tr_{\Ga_n,\Ga_\infty}^{(2)}(R_{\Ga_\infty}(f))$ does not depend on $n$.
In this case it equals the $L^2$-trace $\tr_{\Ga,\Ga_\infty}^{(2)}(R_{\Ga_\infty}(f))$ with respect to $\Ga$, 
$$
\tr_{\Ga,\Ga_\infty}^{(2)}(R_{\Ga_\infty}(f))=\frac1{\vol(\Ga\bs G)}\sum_{[\ga]_{\Ga}\subset \Ga_\infty}
\vol\(\Ga_{\ga}\bs G_\ga\)\,\CO_\ga(f).
$$
\end{lemma}

\begin{proof}
For $\phi\in L^2(\Ga_\infty\bs G)$ we have
\begin{align*}
R_{\Ga_\infty}(f)\phi(x)&= \int_G f(y)\phi(xy)\,dy
=\int_Gf(x^{-1}y)\phi(y)\,dy\\
&= \int_{\Ga_\infty\bs G}\underbrace{\sum_{\ga\in\Ga_\infty}f(x^{-1}\ga y)}_{=k_f(x,y)}\phi(y)\,dy\\
&=\sum_{v\in V_n}\int_{v\CF_n}k_f(x,y)\phi(y)\,dy
=\int_{\CF_n}\sum_{v\in V_n}k_f(x,vy)\phi(vy)\,dy.
\end{align*}
Now, as $\Ga_\infty$ is normal,
\begin{align*}
k_f(x,vy)=\sum_{v\in V_n}\sum_{\ga\in\Ga_\infty}f(x^{-1}\ga vy)&=\sum_{v\in V_n}\sum_{\ga\in\Ga_\infty}f(x^{-1}v\ga y)=k_f(v^{-1}x,y), 
\end{align*}
so that 
\begin{align*}
R_{\Ga_\infty}(f)\phi(x)&=\int_{\CF_n}\sum_{v\in V_n}k_f(v^{-1}x,y)\phi(vy)\,dy.
\end{align*}
As the trace of an integral operator on $L^2(\Ga_n\bs G)$ is given by the integral over the diagonal of the kernel (Proposition 9.3.1 of \cite{HA2}), we get
\begin{align*}
\tr_{\Ga_n,\Ga_\infty}^{(2)}(R_{\Ga_\infty}(f))&=\frac1{\vol(\Ga_n\bs G)}
\int_{\CF_n}k_f(x,x)\,dx
=\frac1{\vol(\Ga_n\bs G)}
\int_{\Ga_n\bs G}\sum_{\ga\in\Ga_\infty} f(x^{-1}\ga x)\,dx\\
&=\frac1{\vol(\Ga_n\bs G)}
\sum_{[\ga]_{\Ga_n}\subset\Ga_\infty}\int_{\Ga_n\bs G}\sum_{\ga'\in[\ga]_{\Ga_n}} f(x^{-1}\ga' x)\,dx\\
&=\frac1{\vol(\Ga_n\bs G)}
\sum_{[\ga]_{\Ga_n}\subset\Ga_\infty}\int_{\Ga_n\bs G}\sum_{\sigma\in\Ga_n/\Ga_{n,\ga}} f((\sigma x)^{-1}\ga \sigma x)\,dx\\
&=\frac1{\vol(\Ga_n\bs G)}
\sum_{[\ga]_{\Ga_n}\subset\Ga_\infty}\int_{\Ga_{n,\ga}\bs G} f( x^{-1}\ga x)\,dx\\
&=\frac1{\vol(\Ga_n\bs G)}
\sum_{[\ga]_{\Ga_n}\subset\Ga_\infty}\vol(\Ga_{n,\ga}\bs G_\ga)\,\CO_\ga(f).\\
\end{align*}
Now assume all $\Ga_n$ are normal in some $\Ga$.
Then
\begin{align*}
\tr_{\Ga_n,\Ga_\infty}^{(2)} R_{\Ga_\infty}(f)&=\sum_{[\ga]_\Ga\subset\Ga_\infty}
\sum_{[\ga]_{\Ga_n}\subset [\ga]_\Ga}
\frac{\vol(\Ga_{n,\ga}\bs G_\ga)}{\vol(\Ga_n\bs G)}\,\CO_\ga(f)\\
&=\sum_{[\ga]_\Ga\subset\Ga_\infty}
\sum_{[\ga]_{\Ga_n}\subset [\ga]_\Ga}
\frac{[\Ga_\ga:\Ga_{n,\ga}]}{[\Ga:\Ga_n]}
\frac{\vol(\Ga_{\ga}\bs G_\ga)}{\vol(\Ga\bs G)}\,\CO_\ga(f)\\
&=\sum_{[\ga]_\Ga\subset\Ga_\infty}
\underbrace{\#\([\ga]_\Ga/\Ga_n\)
\frac{[\Ga_\ga:\Ga_{n,\ga}]}{[\Ga:\Ga_n]}}_{=1}
\frac{\vol(\Ga_{\ga}\bs G_\ga)}{\vol(\Ga\bs G)}\,\CO_\ga(f).\mqed
\end{align*}
\end{proof}

\begin{proposition}\label{prop4.3}
Let $(\Ga_n)$ be a sequence of cocompact lattices.
Let $\Ga_\infty$ be a common normal subgroup.
Then $\Ga_n\tto{Pl}\Ga_\infty$ if and only if the sequence
$$
\frac{\tr R_{\Ga_n}(f)}{\vol\(\Ga_n\bs G\)}-\tr_{\Ga_n,\Ga_\infty}^{(2)} R_{\Ga_\infty}(f)
$$
tends to zero for every $f\in C_c^\infty(G)$.

\end{proposition}

\begin{proof}
The difference $\frac{\tr R_{\Ga_n}(f)}{\vol\(\Ga_n\bs G\)}-\tr_{\Ga_n,\Ga_\infty}^{(2)} R_{\Ga_\infty}(f)$ equals
$$
\frac1{\vol(\Ga_n\bs G)}
\sum_{[\ga]_{\Ga_n}\subset(\Ga_n\sm\Ga_\infty)}\vol(\Ga_{n,\ga}\bs G_\ga)\,\CO_\ga(f).
$$
The same computation as in the proof of Lemma \ref{lem4.2} shows that this equals
$$
\frac1{\vol(\Ga_n\bs G)}
\sum_{[\ga]_{\Ga_n}\subset(\Ga_n\sm\Ga_\infty)}\int_{\Ga_n\bs G}\sum_{\ga'\in[\ga]_{\Ga_n}} f(x^{-1}\ga' x)\,dx.
$$
Form here the claim is clear.
\end{proof}

\begin{theorem}
Let $(\Ga_n)$ be a sequence of cocompact lattices.
Let $\Ga_\infty$ be a common normal subgroup and assume each $\Ga_n$, $1\le n\le\infty$, is normal in some lattice $\Ga$.
Then the following are equivalent:
\begin{enumerate}[\rm (a)]
\item  $(\Ga_n)$ is BS-convergent to $\Ga_\infty$.
\item
$(\Ga_n)$ is Plancherel convergent to $\Ga_\infty$.
\item
For each $f\in C_c^\infty(G)$ one has 
$$
\frac{\tr R_{\Ga_n}(f)}{\vol(\Ga_n\bs G)}\ \longrightarrow\ \tr_{\Ga,\Ga_\infty}^{(2)}(R_\Ga(f)),
$$
as $n\to\infty$.
\end{enumerate}
\end{theorem}

\begin{proof}
The equivalence (a)$\Leftrightarrow$(b) follows from Theorem \ref{thm3.3}, as the condition that all $\Ga_n$ lie in one lattice $\Ga$ implies that the sequence $(\Ga_n)$ is uniformly discrete.

(b)$\Leftrightarrow$(c) is Proposition \ref{prop4.3} together with Lemma \ref{lem4.2}.
\end{proof}

Suppose now that $G$ is second countable and type I.
Then $G$ acts on $L^2(\Ga_\infty\bs G)$ and this action defines a unitary representation of $G$.
 Theorem 8.6.6 of \cite{Dix} says that there are mutually singular measures $\what\mu_{\Ga_\infty,1},\what\mu_{\Ga_\infty,2},\dots$ and $\what\mu_{\Ga_\infty,\infty}$ on the unitary dual $\what G$ such that
\begin{align*}
L^2(\Ga_\infty\bs G)\cong \what{\bigoplus_{1\le j\le\infty}}\ell^2(j)\ \what\otimes\int_{\what G}\zeta\,d\what\mu_{\Ga_\infty,j}(\zeta),
\end{align*}
where $\ell^2(\infty)=\ell^2(\N)$ and the hats over the sum and the product refer to  Hilbert space completions.
The measures $\mu_j$ are unique up to equivalence.
We choose representatives and  write $H(\mu_j)$ for the Hilbert space of $L^2$-sections of the canonical Hilbert bundle over $\supp(\mu_j)$.
Then for every $\phi\in L^2(\Ga_\infty\bs G)$ one has
\begin{align*}
\norm\phi^2=\sum_{j\in\N\cup\{\infty\}}\norm{\phi_j}^2_{\ell^2(j)\what\otimes H(\mu_j)},
\end{align*}
where $\phi=\sum_j\phi_j$ is the above decomposition.

\begin{definition}
We  say that
$$
\what\mu_H=\sum_{j\in\N\cup\{\infty\}}j\mu_j
$$
is the \e{spectral measure} attached to $H$.
Note that $\what\mu_{\{1\}}$ is in general different from the Plancherel measure of $G$, as the latter neglects multiplicities.
\end{definition}

Let $\CA$ be the von Neumann algebra acting on $L^2(\Ga_\infty\bs G)$ defined as the commutant of the left translation action of $\Ga$.
Then $\CA=\CB\(L^2(\Ga\bs G)\)\what\otimes\ell^2(\Ga/\Ga_\infty)$ and the  $L^2$-trace $\tr_{\Ga,\Ga_\infty}^{(2)}$ is defined on $\CA$.
Let $\CB\subset\CA$ be the von Neumann subalgebra acting on $L^2(\Ga_\infty\bs G)$ which is generated by the right translation action $R(G)$ of $G$.
Then $\CB$ respects the direct integral decomposition of $L^2(\Ga_\infty\bs G)$ and 
there exist traces $\tau_{\Ga_\infty,\pi}$, $\pi\in\what G$ such that for $f\in C_c^\infty(G)$ one has
$$
\tr_{\Ga,\Ga_\infty}^{(2)}(R_\Ga(f))
=\sum_{1\le j\le\infty} \int_{\what G}\tau_{\Ga_\infty,\pi}(\pi(f))\,d\what\mu_{\Ga_\infty,j}(\pi).
$$
In the next section we shall compute the  measure $\what\mu_{\Ga_\infty,j}$ in a special case.

\section{An example of a thin group}

Let $\Ga\subset G=\PSL_2(\R)$ be a torsion-free, cocompact lattice.
Let $\chi:\Ga\to\Z$ be a surjective group homomorphism.
For $n\in\N$ let
$$
\Ga_n=\chi^{-1}(n\Z)
$$
and set $\Ga_\infty=\ker(\chi)$.
Then $\Gamma_\infty$ is a discrete subgroup which is not a lattice, but still Zariski dense in $G$, i.e., a so called thin group \cite{thin}.

Let $\ga_0\in\Ga$ be an element with $\chi(\ga_0)=1$, then $\Ga$ is generated by $\Ga_\infty$ together with $\ga_0$.

For each $\la\in\T=\{z\in\C:|z|=1\}$ we get a group homomorphism $\Ga\to\T$, $\ga\mapsto \la^{\chi(\ga)}$.
Let
$L^2(\Ga\bs G,\la)$ denote the space of all measurable functions (modulo null-functions) $f:G\to\C$ which satisfy
\begin{itemize}
\item $f(\ga x)=\la^{\chi(\ga)}f(x)$,
\item $\displaystyle\int_{\Ga\bs G}|f(x)|^2\,dx<\infty$.
\end{itemize}
Let $R_\la$ denote the unitary representation of $G$ given by right translation on the Hilbert space $L^2(\Ga\bs G,\la)$.
For any given $f\in C_c^\infty(G)$ the operator $R_\la(f)$ is trace class \cite{HA2} and one has
$$
\tr R_\la(f)=\sum_{[\ga]}\la^{\chi(\ga)}\vol(\Ga_\ga\bs G_\ga)\CO_\ga(f).
$$
This is a finite sum of the form $\sum_{n=-N}^Nc_n \la^n$ for some coefficients $c_l$ independent of $\la$.
Therefore the family $(R_\la)_{\la\in\T}$ depends continuously on $\la$ in the sense that there is maps $\pi_1,\pi_2,\dots$ from $\T$ to $\what G$ with the property that for each $j$ the set of $K$-types of $\pi_j(\la)$ does not depend on $\la$ and the Casimir eigenvalue $\pi_j(\la)(\Om)$ depends on $\la$ as a continuous function.
In particular, the family $(L^2(\Ga\bs G,\chi_\la))_{\la\in\T}$ forms a Hilbert-bundle and one thus gets a unitary representation $R_\T$ of $G$ on the section space
$$
\int_{\T}^\oplus L^2(\Ga\bs G,\la)\,d\la,
$$
where the integral is with respect to the normalised Haar measure $d\la$ on the group $\T$.

\begin{theorem}
As a unitary $G$-representation, the right translation representation of $G$ on $L^2(\Ga_\infty\bs G)$ is isomorphic with the direct Hilbert integral
$$
\int_{\T}^\oplus L^2(\Ga\bs G,\chi_\la)\,d\la
$$
or, equivalently,
$$
\bigoplus_{j=1}^\infty \int_\T^\oplus \pi_j(\la)\,d\la.
$$
\end{theorem}

\begin{proof}
We define a map $\Phi:\int_{\T}^\oplus L^2(\Ga\bs G,\chi_\la)\,d\la\to L^2(\Ga_\infty\bs G)$ by setting
$$
\Phi(s)(x)=\int_\T s_\la(x)\,d\la.
$$
Then $f=\Phi(s)$ is $\Ga_\infty$ left-invariant and the so defined map $\Phi$ is $G$-equivariant.
We show that $\Phi$ is an isometry.
Fixing a fundamental set $\CF$ for $\Ga$ we compute
\begin{align*}
\norm{\Phi(s)}^2&=\int_{\Ga_\infty\bs G}|f(x)|^2\,dx= \int_{\Ga_\infty\bs G}\left|\int_\T s_\la(x)\,d\la\right|^2\,dx\\
&=\sum_{k\in\Z}\int_\CF\left|\int_\T s_\la(\ga_0^kx)\,d\la\right|^2\,dx\\
&=\int_\CF\sum_{k\in\Z}\left|\int_\T \la^ks_\la(x)\,d\la\right|^2\,dx
=\int_\CF\int_\T\left|s_\la(x)\right|^2\,d\la\,dx= \norm s^2,
\end{align*}
where in the last line we have used the Plancherel formula for the circle group $\T$.

Finally we need to show that $\Phi$ is surjective.
For this let $f\in L^2(\Ga_\infty\bs G)$ be orthogonal to the image of $\Phi$.
Then for any $s\in \int_{\T}^\oplus L^2(\Ga\bs G,\chi_\la)\,d\la$ we have
\begin{align*}
0&=\sp{f,\Phi(s)}\\
&= \int_{\Ga_\infty\bs G}f(x)\int_\T\ol{s_\la(x)}\,d\la\,dx\\
&=\sum_{k\in\Z}\int_\CF f(\ga_0^kx)\int_\T\ol{s_\la(\ga_0^k x)}\,d\la\,dx\\
&=\int_\T\sum_{k\in\Z}\int_\CF f(\ga_0^kx)\ol{s_\la(x)}\,dx\,\ol{\la}^k\,d\la.
\end{align*}
We can prescribe $s_\la$ to be $\la$-independent on $\CF$. Then we get
$$
0=\int_\CF f(x)\ol{s(x)}\,dx.
$$
On $\CF$ we can choose $s$ arbitrarily, so that we conclude $f=0$.
\end{proof}

We can get some more detailed information on the direct integrals occurring in the theorem.
For each $\la\in\T$ we define a line bundle $E_\la$ over $\Ga\bs G$ by
$$
E_\la=\Ga\bs (G\times\C),
$$
where $\Ga$ acts on $G\times\C$ via
$\ga(x,z)=(\ga x,\la^{-\chi(\ga)}z)$.
The union $E=\bigcup_{\la\in\T}E_\la$ can be viewed as a line bundle over $\T\times (\Ga\bs G)$.
As such, it has a natural analytic structure given as follows: There exists $f\in C^\om(G)$ an analytic function on $G$ such that the sum
$$
s_\la(x)=\sum_{\ga\in\Ga}\la^{-\chi(\ga)}f(\ga x)
$$ 
converges absolutely in a neighborhood of $\T\times G$ inside its complexification $\C^\times\times G_\C$.
One gets such $f$ for instance  a gaussian function in Iwasawa coordinates.
Sections of this kind also give local trivialisations of the line bundle $E$.
Note that under this trivialisation, the Casimir operator $\Om$ turns into an analytic family $(\Om_\la)_{\la\in\T}$ of operators on $L^2(\Ga\bs G)$ given by
$$
\Om_\la(f)=s_\la^{-1}\Om(s_\la f).
$$

\begin{proposition}
The sections $\pi_j:\T\to \bigcup_{\la\in\T}L^2(\Ga\bs G,\la)$ can be chosen to be analytic.
Then the Casimir eigenvalues $\la\mapsto\pi_j(\la)(\Om)$ are analytic maps on $\T$.
For each $j$, the direct integral $\int_\T^\oplus\pi_j(\la)\,d\la$ is either an infinite multiple of an irreducible representation or does not have any irreducible sub-representation at all.
\end{proposition}

\begin{proof}
After taking analytic local trivialisations as given above, the first assertion follows from the theory of analytic families of operators as in \cite{Kato}. 
As the map $\eta:\la\mapsto\pi_j(\la)(\Om)$ is analytic, it is either constant, which yields the first case, or it is nowhere constant and each value $z\in\T$ has only finitely many pre-images, i.e., $\eta^{-1}(z)$ is a set of measure zero in $\T$.
\end{proof}

{\bf Question.} 
If for a given $j$, any $\pi_j(\la)$ is a discrete series representation, then $\pi_j$ is the constant map and we are in the first case of the theorem.
It is not clear whether this is the only possibility, i.e., if for some $\la$ the representation $\pi_j(\la)$ is a principal series or complementary series, or a Steinberg representation, is it true that $\pi_j$ is not constant?

{\small Mathematisches Institut,
Auf der Morgenstelle 10,
72076 T\"ubingen,
Germany\\
\tt deitmar@uni-tuebingen.de}

\end{document}